\documentclass[11pt]{article}
\usepackage{amsmath,amssymb,amsthm}
\hfuzz=\maxdimen
\newtheorem{thm}{Theorem}[section]

\newtheorem{lem}[thm]{Lemma}

\newtheorem{conj}[thm]{Conjecture}

\theoremstyle{definition}
\newtheorem{exam$s$-union ple}[thm]{Example}

\begin{document}
\date{}

\title{Suboptimal $s$-union familes and  $s$-union antichains for vector spaces\footnote{This research was supported by  National Natural Science Foundation of China (12171028) and Beijing Natural
Science Foundation (1222013).}}

\author{
 {\small Yunjing  Shan,}  {\small Junling  Zhou}\\
{\small School of Mathematics and Statistics}\\ {\small Beijing Jiaotong University}\\
  {\small Beijing  100044, China}\\
 {\small jlzhou@bjtu.edu.cn}\\
}

\maketitle

\begin{abstract}
  Let $V$ be an $n$-dimensional vector space over the finite field $\mathbb{F}_{q}$, and let $\mathcal{L}(V)=\bigcup_{0\leq k\leq n}\left[V\atop k\right]$ be the set of all subspaces of $V$. A family of subspaces $\mathcal{F}\subseteq \mathcal{L}(V)$ is $s$-union if ${\rm dim}$$(F+F')\leq s$ holds for all $F$, $F'\in\mathcal{F}$. A family $\mathcal{F}\subseteq \mathcal{L}(V)$ is an antichain if $F\nleq F'$ holds for any two distinct $F, F'\in \mathcal{F}$. The optimal $s$-union families in $\mathcal{L}(V)$ have been determined by Frankl and Tokushige in $2013$. The upper bound of cardinalities of $s$-union $(s<n)$ antichains in $\mathcal{L}(V)$ has been established by Frankl recently, while the structures of optimal ones have not been displayed. The present paper determines all suboptimal $s$-union families for vector spaces and then investigates $s$-union antichains. For $s=n$ or $s=2d<n$, we determine all optimal and suboptimal $s$-union antichains completely. For $s=2d+1<n$, we prove that an optimal antichain is either $\left[V\atop d\right]$ or contained in $\left[V\atop d\right]\bigcup \left[V\atop d+1\right]$ which satisfies an equality related with shadows.
\end{abstract}


\section{Introduction}

Let $X$ be an $n$-element set and let $\tbinom{X}{k}$ denote the set of all $k$-element subsets of $X$. For the power set of $X$, we use the notation $2^{X}$. We say that a family of subsets $\mathcal{F}\subseteq 2^{X}$ is {\it$s$-union}~if~$|F \cup F'|\leq s$ holds for all $F$, $F'\in\mathcal{F}$. A family $\mathcal{F}$ is called {\it $t$-intersecting} if for all $F$, $F'\in\mathcal{F}$, we have $|F\cap F'|\geq t$. Since $\mathcal{F}\subseteq 2^{X}$ is $s$-union if and only if $\{X\setminus F:F\in \mathcal{F}\}$ is an $(n-s)$-intersecting family, the two concepts are essentially the same.

Let $\mathcal{F}\subseteq 2^{X}$ be $s$-union. An $s$-union family is said to be {\it optimal} if it has the largest possible cardinality. It is obvious that $\mathcal{F}=2^{X}$ is the optimal $s$-union family if $s=n$. For $s=n-1$, there are many optimal $s$-union families achieving the maximum cardinality $2^{n-1}$ {{\rm\cite{ekr set}}}. For $s\leq n-2$, Katona {{\rm\cite{katona set}}} showed that $|\mathcal{F}| \leq \sum_{i=0}^d \tbinom{n}{i}$ if $s=2d$, and $|\mathcal{F}| \leq \sum_{i=0}^d \tbinom{n}{i}+\tbinom{n-1}{d}$ if $s=2d+1$. The optimal $s$-union families are proved to be isomorphic to the Katona familes $\mathcal{K}(n,s)$ defined as follows. For $s=2d$, let

\begin{equation*}
\mathcal{K}(n,2d)=\{F\subseteq X:|F|\leq d\}.
\end{equation*}
For $s=2d+1$, let
\begin{equation*}
\mathcal{K}(n,2d+1)=\{F\subseteq X:|F|\leq d\}\bigcup \{F\subseteq X:|F|=d+1, y\in F\},
\end{equation*}
where $y$ is a fixed element of $X$. In 2017, Frankl {{\rm\cite{union}}} investigated the suboptimal $s$-union families $(s<n)$, meaning that they have the largest possible cardinality under the condition that they are not contained in any of the optimal $s$-union families. Frankl established the following theorem.

\begin{thm}{\rm(\cite[Theorem 2]{union})}\label{set}
Let $X$ be an $n$-element set and $2\leq s<n$. Suppose that $\mathcal{F}\subseteq 2^{X}$ is $s$-union and $\mathcal{F}\nsubseteq \mathcal{K}(n,s)$. Then the following hold.
\begin{itemize}
\item[\rm(i)] When $s=2d,$
\begin{equation*}
|\mathcal{F}| \leq \sum_{i=0}^d \tbinom{n}{i}-\tbinom{n-d-1}{d}+1.
\end{equation*}
Moreover for $s\leq n-2$, equality holds if and only if
\begin{equation*}
\mathcal{F}=(\{F\subseteq X:|F|\leq d\}\setminus\{F\in \tbinom{X}{d}:F\cap D =\emptyset\})\bigcup \{D\},
\end{equation*}
where $D$ is a fixed $(d+1)$-subset of $X$.
\item[\rm(ii)] When $s=2d+1,$
\begin{equation*}
|\mathcal{F}| \leq \sum_{i=0}^d \tbinom{n}{i}+\tbinom{n-1}{d}-\tbinom{n-d-2}{d}+1.
\end{equation*}
Moreover for $s\leq n-2$, equality holds if and only if
\begin{equation*}
\mathcal{F}=\{F\subseteq X:|F|\leq d\}\bigcup\{F\in \tbinom{X}{d+1}:y\in F,F\cap D \neq \emptyset\}\bigcup\{D\},
\end{equation*}
where $D\in \tbinom{X}{d+1}$, $y \in X$ are fixed with $y\notin D$ or
\begin{equation*}
\mathcal{F}=\{F\subseteq X:|F|\leq 2\}\bigcup \{F\in \tbinom{X}{3}:|F\cap D|\geq 2\},
\end{equation*}
where $D$ is a fixed $3$-subset of $X$ if $s=5$.
\end{itemize}
\end{thm}

The problems in extremal set theory have natural extensions to families of subspaces over a finite field. Throughout the paper we always let $V$ be an $n$-dimensional vector space over the finite field $\mathbb{F}_{q}$. Let $\left[V\atop k\right]$ denote the family of all $k$-dimensional subspaces of $V$. For $m\in \mathbb{R}, k\in \mathbb{Z}^{+}$, define the Gaussian binomial coefficient by
$$\left[m\atop k\right]_{q}:=\prod_{0\leq i<k}\dfrac{q^{m-i}-1}{q^{k-i}-1}.$$
Obviously, the size of $\left[V\atop k\right]$ is $\left[n\atop k\right]_{q}$. If $k$ and $q$ are fixed, then $\left[m\atop k\right]_{q}$ is a continuous function of $m$ which is positive and strictly increasing when $m\geq k$. If there is no ambiguity, the subscript $q$ can be omitted. Let
\begin{equation*}
\mathcal{L}(V)=\bigcup_{0\leq k\leq n}\left[V\atop k\right].
\end{equation*}

For any two subspaces $A$, $B\in \mathcal{L}(V)$, let $A\leq B$ denote that $A$ is a subspace of $B$ and $A+B$ denote the linear span of $A$, $B$. Let $\mathcal{F}\subseteq \mathcal{L}(V)$ be a family of subspaces, we say that $\mathcal{F}$ is {\it$s$-union} if ${\rm dim}$$(F+F')\leq s$ holds for all $F$, $F'\in\mathcal{F}$. A family $\mathcal{F}\subseteq \mathcal{L}(V)$ is called {\it$t$-intersecting} if for all $F$, $F'\in\mathcal{F}$, we have ${\rm dim}$$(F\cap F')\geq t$. In particular, we say that $\mathcal{F}$ is an {\it intersecting family} if $t=1$. Endow $V$ with the usual inner product $\langle \cdot,\cdot\rangle$. For a subspace $U$ of $V$, let
\begin{equation*}
U^{\bot}=\{v\in V:\langle u, v\rangle=0~{\rm for~all}~u\in U\}
\end{equation*}
be the orthogonal complement of $U$. For a family $\mathcal{F}\subseteq \mathcal{L}(V)$, denote
\begin{equation*}
\mathcal{F}^{\bot}=\{F^{\bot}:F\in \mathcal{F}\}.
\end{equation*}
Since $(F+F')^{\bot}=F^{\bot}\cap F'^{\bot}$ holds for all $F, F'\in \mathcal{F}$, it is obvious that $\mathcal{F}$ is $s$-union if and only if $\mathcal{F}^{\bot}$ is $(n-s)$-intersecting.

Obviously, $\mathcal{F}=\mathcal{L}(V)$ is the optimal $s$-union family if $s=n$. For $s=n-1$, the optimal $s$-union families in $\mathcal{L}(V)$ were determined by Blokhuis et al. \cite{hm vector}. For $s\leq n-2$, we record the optimal families $\mathcal{K}[n,s]$ as follows, which were displayed in the form of $t$-intersecting families in~\cite{katona}. For $s=2d$, let
\begin{equation*}
\mathcal{K}[n,2d]=\{F\leq V:{\rm dim}(F)\leq d\}.
\end{equation*}
For $s=2d+1$, let
\begin{equation*}
\mathcal{K}[n,2d+1]=\{F\leq V:{\rm dim}(F)\leq d\}\bigcup \{F\in \left[V\atop d+1\right]:E\leq F\},
\end{equation*}
where $E$ is a fixed $1$-dimensional subspace of $V$.

\begin{thm}{\rm({\cite[Theorem 4]{katona}})}\label{space one}
Suppose that $\mathcal{F}\subseteq \mathcal{L}(V)$ is $s$-union, $2\leq s<n$. Then the following hold.
\begin{itemize}
\item[\rm(i)] When $s=2d,$
\begin{equation*}
|\mathcal{F}| \leq \sum_{i=0}^d \left[n\atop i\right].
\end{equation*}
Moreover, for $s\leq n-2$, equality holds if and only if $\mathcal{F}=\mathcal{K}[n,2d]$.
\item[\rm(ii)] When $s=2d+1,$
\begin{equation*}
|\mathcal{F}| \leq \sum_{i=0}^d \left[n\atop i\right]+\left[n-1\atop d\right].
\end{equation*}
Moreover, for $s\leq n-2$, equality holds if and only if $\mathcal{F}=\mathcal{K}[n,2d+1]$.
\end{itemize}

\end{thm}

Excluding the optimal families provided in Theorem \ref{space one}, we will consider the suboptimal $s$-union families in $\mathcal{L}(V)$ and establish a vector space version of Theorem \ref{set}. Let us define the families $\mathcal{T}[n,s]$. For $s=2d$, define
\begin{equation*}
\mathcal{T}[n,2d]=(\{F\leq V:{\rm dim}(F)\leq d\}\setminus\{F\subseteq \left[V\atop d\right]:{\rm dim}(F\cap U)=0\})\bigcup\{U\},
\end{equation*}
where $U$ is a fixed $(d+1)$-dimensional subspace of $V$.

For $s=2d+1$, define
\begin{equation*}
\mathcal{T}[n,2d+1]=\{F\leq V:{\rm dim}(F)\leq d\}\bigcup \{F\in \left[V\atop d+1\right]:E\leq F, {\rm dim}(F\cap U)\geq 1\}\bigcup \left[E+U\atop d+1\right],
\end{equation*}
 where $E\in \left[V\atop 1\right]$ and $U\in \left[V\atop d+1\right]$ are fixed subspaces of $V$ with $E\nleq U$.

For $s=5$, define also
\begin{equation*}
\mathcal{J}[n,5]=\{F\leq V:{\rm dim}(F)\leq 2\}\bigcup \{F\in \left[V\atop 3\right]:{\rm dim}(F\cap D)\geq 2\},
\end{equation*}
where $D$ is a fixed $3$-dimensional subspace of $V$.

A main result of this paper in the next theorem shows that $\mathcal{T}[n,2d]$, $\mathcal{T}[n,2d+1]$ and $\mathcal{J}[n,5]$~(for~$s=5)$~are suboptimal $s$-union families.

\begin{thm}\label{space two}
Let $2\leq s<n$. Suppose that $\mathcal{F}\subseteq \mathcal{L}(V)$ is $s$-union and $\mathcal{F}\nsubseteq \mathcal{K}[n,s]$. Then the following hold.
\begin{itemize}
\item[\rm(i)] If $s=2d,$ then
\begin{equation}\label{eq8}
\begin{array}{rl}
|\mathcal{F}| \leq \sum_{i=0}^d\left[n\atop i\right]-q^{d(d+1)}\left[n-d-1\atop d\right]+1.
\end{array}
\end{equation}
Moreover for $s\leq n-2$, equality holds if and only if $\mathcal{F}=\mathcal{T}[n,2d].$

\item[\rm(ii)] If $s=2d+1,$ either $q\geq 3$ and $n\geq 2d+3, or~q = 2~ and~ n\geq2d+4$, then
\begin{equation}\label{eq9}
\begin{array}{rl}
|\mathcal{F}| \leq \sum_{i=0}^d \left[n\atop i\right]+\left[n-1\atop d\right]-q^{d(d+1)}\left[n-d-2\atop d\right]+q^{d+1}.
\end{array}
\end{equation}
Moreover for $s\leq n-2$, equality holds if and only if $\mathcal{F}=\mathcal{T}[n,2d+1]$ or alternatively $\mathcal{F}=\mathcal{J}[n,5]~if~s=5$.
\end{itemize}

\end{thm}

A family $\mathcal{F}\subseteq \mathcal{L}(V)$ is an {\it antichain} if $F\nleq F'$ holds for any two distinct $F, F'\in \mathcal{F}$. In 2021, Frankl \cite{sperner} obtained the upper bound of the cardinalities of $s$-union antichains for vector spaces.
\begin{thm}\label{antichain1}{\rm(\cite[Theorem 3.5]{sperner})}
If $\mathcal{F}\subseteq \mathcal{L}(V)$ is an $s$-union antichain with $2\leq s< n$, then
\begin{equation}\label{anti}
|\mathcal{F}|\leq \left[n\atop \lfloor\frac{s}{2}\rfloor\right].
\end{equation}
\end{thm}


The second main objective of the paper is to determine the structures of all optimal $s$-union antichains and then the suboptimal ones. Let us define the families $\mathcal{A}[n,s]$, $\mathcal{B}[n,s]$. For $2\leq s\leq n$, define
\begin{equation*}
\mathcal{A}[n,s]=(\left[V\atop \lceil \frac{s}{2}\rceil\right]\setminus\{F\in \left[V\atop \lceil \frac{s}{2}\rceil\right]:U\leq F\})\bigcup \{U\},
\end{equation*}
where $U$ is a fixed $(\lceil \frac{s}{2}\rceil-1)$-dimensional subspace of $V$; define
\begin{equation*}
\mathcal{B}[n,s]=(\left[V\atop \lfloor\frac{s}{2}\rfloor\right]\setminus\{F\in \left[V\atop \lfloor\frac{s}{2}\rfloor\right]:F\leq W\})\bigcup \{W\},
\end{equation*}
where $W$ is a fixed $(\lfloor\frac{s}{2}\rfloor+1)$-dimensional subspace of $V$.

It is obvious that an $s$-union antichain is just an antichain if~$s=n$. We determine the structures of all optimal and suboptimal antichains in this paper.




\begin{thm}\label{2d+1}
Let $\mathcal{F}\subseteq \mathcal{L}(V)$ be an antichain and $n>1$. Then the following hold.
\begin{itemize}
\item[\rm(i)]~
$|\mathcal{F}|\leq \left[n\atop \lfloor\frac{n}{2}\rfloor\right]$. Moreover, equality holds if and only if $\mathcal{F}=\left[V\atop \lfloor\frac{n}{2}\rfloor\right]$ or $\mathcal{F}=\left[V\atop \lceil\frac{n}{2}\rceil\right]$.
\item[\rm(ii)]~If $\mathcal{F}\nsubseteq \left[V\atop \lfloor\frac{n}{2}\rfloor\right]$ and $\mathcal{F}\nsubseteq \left[V\atop \lceil\frac{n}{2}\rceil\right]$, then
\begin{equation*}
|\mathcal{F}|\leq \left[n\atop \lfloor\frac{n}{2}\rfloor\right]-q\left[\lfloor\frac{n}{2}\rfloor\atop 1\right].
\end{equation*}
Moreover, equality holds if and only if $\mathcal{F}=\mathcal{A}[n,n]$ or $\mathcal{B}[n,n]$.
\end{itemize}
\end{thm}
For $0\leq u \leq n$, let us define the {\it$u$-shadow} of $\mathcal{H}\subseteq \mathcal{L}(V)$ by
\begin{equation*}
\bigtriangleup_u(\mathcal{H})=\{G\in \left[V\atop u\right]:G\leq H~{\rm for~some}~H\in \mathcal{H}\}.
\end{equation*}
In particular, the $(u-1)$-shadow of the family $\mathcal{H}\subseteq \left[V\atop u\right]$ is denoted by  $\bigtriangleup(\mathcal{H})$ for convenience.

For $s$-union antichains with $s<n$, we establish another main theorem as follows.

\begin{thm}\label{antichain2}
Let $\mathcal{F}\subseteq \mathcal{L}(V)$ be an $s$-union antichain with $s< n$. Then the following hold.
\begin{itemize}
\item[\rm(i)]~
$|\mathcal{F}|\leq \left[n\atop \lfloor\frac{s}{2}\rfloor\right]$. Moreover, equality holds if and only if either $(a)$ or $(b)$ holds.\\
$(a)$ $\mathcal{F}=\left[V\atop \lfloor\frac{s}{2}\rfloor\right]$; \\
$(b)$ $\mathcal{F}=\mathcal{F}_{d}\bigcup \mathcal{F}_{d+1}$ for $s=2d+1$, where $\mathcal{F}_{d+1}\subseteq \left[V\atop d+1\right], \mathcal{F}_{d}=\left[V\atop d\right]\setminus \bigtriangleup(\mathcal{F}_{d+1})$ and $|\bigtriangleup(\mathcal{F}_{d+1})|=|\mathcal{F}_{d+1}|$.
\item[\rm(ii)]~Suppose $s=2d$ and $\mathcal{F}\nsubseteq \left[V\atop d\right]$. Then $(a)$ or $(b)$ holds.\\
$(a)$ If $d=1$, then
\begin{equation*}
|\mathcal{F}|\leq \left[n\atop 1\right]-q.
\end{equation*}
 Moreover, equality holds if and only if $\mathcal{F}=\mathcal{B}[n,2]$.\\
$(b)$~If $d\geq 2$, then
\begin{equation*}
|\mathcal{F}|\leq \left[n\atop d\right]-q\left[n-d\atop 1\right].
\end{equation*}
 Moreover, equality holds if and only if $\mathcal{F}=\mathcal{A}[n,2d]$.
\end{itemize}
\end{thm}

 The main objective of the paper is to prove Theorems \ref{space two}, \ref{2d+1} and \ref{antichain2}. For Theorem \ref{antichain2}, we first consider the case $d=1$, $s=2<n$. It is obvious that if $d=1$, the optimal $2$-union antichain is $\left[V\atop 1\right]$. Let $\mathcal{F}$ be a suboptimal $2$-union antichain. We can easily find that if $|{\cal F}|> 1$ then any $i$-dimensional subspace with $i=0$ or $i\geq 3$ does not belong to $\mathcal{F}$ and that $|\mathcal{F}\bigcap\left[V\atop 2\right]|\leq1$. Thus $\mathcal{F}=\mathcal{B}[n,2]$ and  $|\mathcal{F}|=\left[n\atop 1\right]-\left[2\atop 1\right]+1=\left[n\atop 1\right]-q$.

\section{ Preliminaries}
In this section, we recall a number of basic theorems and establish several new lemmas in the vector spaces, which are essential for our proofs. Firstly, we introduce the celebrated Erd\H{o}s-Ko-Rado theorem and Hilton-Milner theorem for vector spaces.

\begin{thm}\label{EKR}{\rm(\cite[Theorem 1]{ekr vector},~\cite[Theorem 3]{ekrbc})}
Let $1\leq t\leq k$. Suppose $\mathcal{H}\subseteq \left[V\atop k\right]$ is a $t$-intersecting family. Then we have

\begin{equation*}
|\mathcal{H}|\!\leq\!
\begin{cases}
\left[n-t\atop k-t\right], &\text{if~$n\geq 2k$},\\[.3cm]
\left[2k-t\atop k\right], &\text{if~$2k-t<n\leq2k$}.
\end{cases}\
\end{equation*}
Moreover, equality holds if and only if  one of the following holds:
\begin{itemize}
\item[\rm(i)]~If $n>2k$, then $\mathcal{H}=\{H\in\left[V\atop k\right]:T\leq H\}$ for some $T\in\left[V\atop t\right]$.
\item[\rm(ii)]~If $2k-t<n<2k$, then $\mathcal{H}=\left[Y\atop k\right]$ for some $Y\in\left[V\atop 2k-t\right]$.
\item[\rm(iii)]~If $n=2k$, then $\mathcal{H}=\{H\in\left[V\atop k\right]:T\leq H\}$ for some $T\in\left[V\atop t\right]$ or $\mathcal{H}=\left[Y\atop k\right]$ for some $Y\in\left[V\atop 2k-t\right]$.
\end{itemize}
\end{thm}

\begin{thm}\label{HM}{\rm(\cite[Theorem 1.4]{hm vector})}
Suppose $k\geq2$, and either $q\geq 3$ and $n\geq 2k+1, or~q = 2 ~and~ n\geq2k+2$. Let $\mathcal{H}\subseteq \left[V\atop k\right]$ be an intersecting family with ${\rm dim}$$(\bigcap _{H\in \mathcal{H}}H)=0$. Then
\begin{equation*}
\begin{array}{rl}
|\mathcal{H}|\leq \left[n-1\atop k-1\right]-q^{k(k-1)}\left[n-k-1\atop k-1\right]+q^{k}.
\end{array}
\end{equation*}
Moreover, equality holds if and only if
\begin{equation*}
\mathcal{H}=\{H\in \left[V\atop k\right]:E\leq H, {\rm dim}(H\cap U)\geq 1\}\bigcup \left[E+U\atop k\right],
\end{equation*}
where $E$, $U$ are fixed $1$-dimensional, $k$-dimensional subspaces of~$V$~(respectively) with $E\nleq U$ or for$~k=3$,
\begin{equation*}
\mathcal{H}=\{F\in \left[V\atop 3\right]: {\rm dim}(F\cap D)\geq 2\},
\end{equation*}
where $D$ is a fixed $3$-dimensional subspace of $V$.
\end{thm}

Shadow is an important notion in extremal set theory. We will make use of the following two vector space version of theorems on shadows.

\begin{thm}\label{shadow intersecting}{\rm(\cite[Theorem 3]{katona})}
Let $1\leq t\leq k \leq n$ and let $\mathcal{H}\subseteq \left[V\atop k\right]$ be t-intersecting. Then for $k-t\leq u \leq k$, we have
\begin{equation}\label{eq1}
\begin{array}{rl}
\frac{|\bigtriangleup_u(\mathcal{H})|}{|\mathcal{H}|}\geq  \frac{\left[2k-t\atop u\right]}{\left[2k-t\atop k\right]}.
\end{array}
\end{equation}
\end{thm}
Note that for $k-t< u <k$, the $\text{RHS}$ of \eqref{eq1} is strictly greater than 1.

\begin{thm}\label{k-1-shadow}{\rm(\cite[Theorem 1.4]{shadow})}
Let $\mathcal{H}\subseteq \left[V\atop k\right]$ and let $m\geq k$ be the real number which satisfies $|\mathcal{H}|=\left[m\atop k\right]$. Then
\begin{equation*}
|\bigtriangleup(\mathcal{H})|\geq \left[m\atop k-1\right].
\end{equation*}
Moreover, equality holds if and only if $\mathcal{H}=\left[M\atop k\right]$ for some $M\in\left[V\atop m\right], m\in \mathbb{Z}^{+}$.
\end{thm}

The following result is an analog of {\rm\cite[Proposition 1]{union}} for finite sets.

\begin{lem}\label{s,i}
Suppose $\mathcal{F}\subseteq \mathcal{L}(V)$ is $s$-union and $2\leq s<n$. Let $\mathcal{F}_i=\mathcal{F}\cap \left[V\atop i\right]$ for $0\leq i\leq \lfloor\frac{s}{2}\rfloor$. Then
\begin{equation}\label{eq4}
\begin{array}{rl}
|\mathcal{F}_i|+|\mathcal{F}_{s+1-i}|\leq \left[n\atop i\right].
\end{array}
\end{equation}
Moreover, for $s\leq n-2$, equality holds if and only if $\mathcal{F}_{s+1-i}=\emptyset.$
\end{lem}

\begin{proof}~When $\mathcal{F}_{s+1-i}=\emptyset$, the theorem holds trivially. We suppose $\mathcal{F}_{s+1-i}\neq \emptyset$ in the following. For a fixed $i\leq \lfloor\frac{s}{2}\rfloor$, define the family
 \begin{equation*}
 \mathcal{H}=\{F^{\bot}:F\in \mathcal{F}_{s+1-i}\}\subseteq \left[V\atop n+i-s-1\right].
 \end{equation*}

We claim that $\bigtriangleup_i(\mathcal{H})\cap \mathcal{F}_{i}=\emptyset$. Suppose $F\in \bigtriangleup_i(\mathcal{H})\cap \mathcal{F}_{i}$. Then there exists $H\in \mathcal{H}$ such that $F\leq H$, $H^{\bot}\in \mathcal{F}_{s+1-i}$. So we have ${\rm dim}(F\cap H^{\bot})=0$, i.e.,$~{\rm dim}(F+H^{\bot})=i+(s+1-i)=s+1$, a contradiction to the $s$-union property of $\mathcal{F}$. By the claim, we have

\begin{equation}\label{40}
|\bigtriangleup_i(\mathcal{H})|+|\mathcal{F}_{i}|\leq \left[n\atop i\right].
\end{equation}

 Since $\mathcal{F}_{s+1-i}\subseteq\mathcal{F}$ is $s$-union, $\mathcal{H}$ is an $(n-s)$-intersecting family.
In Theorem\;\ref{shadow intersecting}, setting $u=i, t=n-s, k=n+i-s-1$, yields the following inequality for $s\leq n-1$:
 \begin{equation*}
 |\bigtriangleup_i(\mathcal{H})|\geq |\mathcal{H}|.
 \end{equation*}
When $s\leq n-2$, since $\mathcal{F}_{s+1-i}\neq \emptyset,$~i.e.,$~\mathcal{H}\neq \emptyset$, we have $|\bigtriangleup_i(\mathcal{H})|> |\mathcal{H}|$ by  \eqref{eq1}. Hence, we have
$|\mathcal{H}|+|\mathcal{F}_{i}|\leq \left[n\atop i\right]$ by \eqref{40}. It is clear that $|\mathcal{H}|=|\mathcal{F}_{s+1-i}|$. Therefore, the desired result follows.
\end{proof}

We introduce a counting formula for vector spaces, which is further interpreted by applying the $q$-analog of inclusion-exclusion principle in {\rm\cite{chen}}.





\begin{lem}\label{lemc21}
Let $Z$ be an $m$-dimensional subspace of the $n$-dimensional vector space $V$ over~$\mathbb{F}_{q}$. For a positive integer $l$ with $m+l\leq n$, let $x$ denote the number of
~$l$-dimensional subspaces~~$W$~of~~$V$~such that ${\rm dim(Z\cap W)}=0$. Then the following hold.
\begin{itemize}
\item[\rm(i)]~$x=q^{lm}\left[n-m\atop l\right]=\sum _{0\leq t\leq \min\{m,l\}}(-1)^{t}q^{\frac{t(t-1)}{2}}\left[m\atop t\right]\left[n-t\atop l-t\right]$.
\item[\rm(ii)]~$x\geq \left[n\atop l\right]-\left[m\atop 1\right]\left[n-1\atop l-1\right].$
\end{itemize}

\end{lem}

\begin{proof}~
$\rm(i)$ This is the result of Propositions 2.2 and 2.3 in {\rm\cite{chen}}.

$\rm(ii)$ Let $a=\min\{m,l\}$, if~$a=1$, then the equality in $\rm(ii)$ holds by $\rm(i)$.
If $a\geq 2$, we have
\begin{equation*}
x=\left[n\atop l\right]-\left[m\atop 1\right]\left[n-1\atop l-1\right]+\sum _{t=2}^a (-1)^{t}q^{\frac{t(t-1)}{2}}\left[m\atop t\right]\left[n-t\atop l-t\right].
\end{equation*}
It suffices to prove that $q^{\frac{t(t-1)}{2}}\left[m\atop t\right]\left[n-t\atop l-t\right]\geq q^{\frac{(t+1)t}{2}}\left[m\atop t+1\right]\left[n-t-1\atop l-t-1\right]$, where $a\geq t\geq2$.
Since $n\geq m+l\geq 2a$, we have
\begin{equation*}
\begin{array}{rl}
\frac{q^{n-t}-1}{q^{l-t}-1}\cdot \frac{q^{t+1}-1}{q^{m-t}-1}&> \frac{(q^{n-t}-1)\cdot (q^{t+1}-1)}{q^{l+m-2t}}\\[.3cm]
&= \frac{q^{n+1}-q^{n-t}-q^{t+1}+1}{q^{l+m-2t}}\\[.3cm]
&>\frac{q^{n+1}-q^{n-2}-q^{a+1}}{q^{l+m-2t}}\\[.3cm]
&>\frac{q^{n}}{q^{l+m-2t}}\\[.3cm]
&>q^{t},
\end{array}
\end{equation*}
which implies the desired result.
\end{proof}

Two families $\mathcal{A}$ and $\mathcal{B}$ in $\mathcal{L}(V)$ are said to be {\it cross-Sperner} if there exist no $A\in \mathcal{A}$ and $B \in \mathcal{B}$ with $A\leq B$ or $B\leq A$.  Wang and Zhang ${\cite{cross}}$ obtained the upper bound of sizes of a pair of cross-Sperner families of finite vector spaces.

\begin{thm}\label{cross-sperner}{\rm(\cite[Theorem 1.5]{cross})}
Let $a,b,t$ be positive integers with $a<b<n$. If $\mathcal{A}\subseteq \left[V\atop a\right]$ and $\mathcal{B}\subseteq \left[V\atop b\right]$ are cross-Sperner, then
\begin{equation*}
|\mathcal{A}|+|\mathcal{B}|\leq \max\{\left[n\atop b\right]-\left[n-a\atop b-a\right]+1,~\left[n\atop a\right]-\left[b\atop a\right]+1\}.
\end{equation*}
Moreover equality holds if and only if  one of the following holds:
 \begin{itemize}
\item[\rm(i)]~$\left[n\atop a\right]\leq\left[n\atop b\right]$ and $\mathcal{A}=\{A\}~for~some~A\in \left[V\atop a\right]$ and $\mathcal{B}=\{B\in \left[V\atop b\right]:A\nleq B\}$;
\item[\rm(ii)]~$\left[n\atop a\right]\geq\left[n\atop b\right]$ and $\mathcal{B}=\{B\}~for~some~B\in \left[V\atop b\right]$ and $\mathcal{A}=\{A\in \left[V\atop a\right]:A\nleq B\}$.
 \end{itemize}
\end{thm}

We say that two families $\mathcal{A}$ and $\mathcal{B}$ in $\mathcal{L}(V)$ are {\it cross-$t$-intersecting} if ${\rm dim}(A\cap B)\geq t$ for all $A\in \mathcal{A}, B\in \mathcal{B}$, where $t\geq 1$. In particular, $\mathcal{A}$ and $\mathcal{B}$ are said to be {\it cross-intersecting} if $t=1$. Wang and Zhang ${\cite{cross}}$ also obtained the upper bound of sizes of cross-$t$-intersecting families.

\begin{thm}\label{cross}{\rm(\cite[Theorem 1.4]{cross})}
Let $n\geq 4$, $a,b,t$ be positive integers with $a,b\geq 2, t<\min\{a,b\}, a+b<n+t$, and $\left[n\atop a\right]\leq \left[n\atop b\right]$. If $\mathcal{A}\subseteq \left[V\atop a\right]$ and $\mathcal{B}\subseteq \left[V\atop b\right]$ are cross-$t$-intersecting, then
\begin{equation}\label{74}
|\mathcal{A}|+|\mathcal{B}|\leq \left[n\atop b\right]- \sum _{i=0}^{t-1} q^{(a-i)(b-i)} \left[a\atop i\right]\left[n-a\atop b-i\right]+1.
\end{equation}
Moreover equality holds if and only if  one of the following holds:
 \begin{itemize}
\item[\rm(i)]~$\mathcal{A}=\{A\}~for~some~A\in \left[V\atop a\right]$ and $\mathcal{B}=\{B\in \left[V\atop b\right]:{\rm dim}(A\cap B)\geq t\}$;
\item[\rm(ii)]~$\left[n\atop a\right]=\left[n\atop b\right]$ and $\mathcal{B}=\{B\}~for~some~B\in \left[V\atop b\right]$ and $\mathcal{A}=\{A\in \left[V\atop a\right]:{\rm dim}(A\cap B)\geq t\}$.
 \end{itemize}

\end{thm}

We will sharpen the upper bound of \eqref{74} in a special case in the following lemma, which will be very useful in the proofs of Theorems\;\ref{space two} and\;\ref{antichain2}.

\begin{lem}\label{cross-}
Suppose $\mathcal{A}\subseteq \left[V\atop k\right]$ and $\mathcal{B}\subseteq \left[V\atop k+1\right]$ are cross-intersecting families. Further suppose $\mathcal{B}$ is $2$-intersecting. Then for $n\geq 2k+1$,
\begin{equation}\label{eq6}
\begin{array}{rl}
|\mathcal{A}|+|\mathcal{B}|\leq \left[n\atop k\right]-q^{k(k+1)}\left[n-k-1\atop k\right]+1.
\end{array}
\end{equation}
Moreover for $n\geq 2k+2$, equality holds if and only if $\mathcal{B}=\{B\}~for~some~B\in \left[V\atop k+1\right]$ and $\mathcal{A}=\{A\in \left[V\atop k\right]:{\rm dim}(A\cap B)\geq 1\}$.
\end{lem}

\begin{proof}~
If $k=1$, since $\mathcal{B}\subseteq \left[V\atop 2\right]$ is $2$-intersecting, then $|\mathcal{B}|=1$ and \eqref{eq6} holds by Lemma~\ref{lemc21}~${\rm(i)}$. We always let $k\geq 2$ in the following proof.

First we consider the case$~n=2k+1$. Since $\left[n\atop k\right]=\left[n\atop k+1\right]$,  we can set $a=k+1, b=k, t=1$ in Theorem\;\ref{cross}. Then \eqref{eq6} holds by \eqref{74}.

Next we let~$n\geq 2k+2$.
Since $\mathcal{B}\subseteq \left[V\atop k+1\right]$ is $2$-intersecting and $n\geq 2k+2$, then by Theorem \ref{EKR},  we have
\begin{equation}\label{eq10}
1\leq |\mathcal{B}| \leq \left[n-2\atop k-1\right].
 \end{equation}

For any $\mathcal{C}\subseteq \left[V\atop k+1\right]$, define $\Gamma(\mathcal{C})=\{T\in \left[V\atop k\right]:{\rm dim}(T\cap C)=0~\text{for~some}~C\in \mathcal{C}\}$. Since ${\rm dim}(A\cap B)\geq 1$ for all $A\in \mathcal{A}$ and $B\in \mathcal{B}$, we have $\mathcal{A}\cap \Gamma(\mathcal{B})=\emptyset$. Hence,

\begin{equation}\label{eq11}
|\mathcal{A}|+|\Gamma(\mathcal{B})|\leq \left[n\atop k\right].
 \end{equation}

Case\ $a$$:$~Suppose $\mathcal{B}=\{B\}\subseteq \left[V\atop k+1\right]$. Then $\Gamma(\mathcal{B})=\{T\in \left[V\atop k\right]:{\rm dim}(T\cap B)=0\}$. Hence we have $|\Gamma(\mathcal{B})|=q^{k(k+1)}\left[n-k-1\atop k\right]$ by Lemma~\ref{lemc21}~${\rm(i)}$, which implies \eqref{eq6} by using \eqref{eq11}. Moreover, the equality in \eqref{eq6} holds if and only if the equality in \eqref{eq11} holds, that is $\mathcal{A}=\{A\in \left[V\atop k\right]:{\rm dim}(A\cap B)\geq 1\}$.

Case\ $b$$:$~Suppose $|\mathcal{B}|\geq 2$. Let
\begin{equation*}
\mathcal{B}=\{B_{1}, B_{2}, \ldots, B_{r}\}\subseteq \left[V\atop k+1\right],
\end{equation*}
where $r\geq 2$. We claim that

\begin{equation}\label{eq31}
|\Gamma(\mathcal{B})|\geq q^{k(k+1)}\left[n-k-1\atop k\right]+q^{(k-1)(k+1)}\left[n-k-2\atop k-1\right].
\end{equation}

Since $\{B_{1}, B_{2}\}\subseteq \mathcal{B}$, then $\Gamma(\{B_{1}, B_{2}\})\subseteq\Gamma(\mathcal{B})$ by the definition of $\Gamma(\mathcal{B})$. Let $\Gamma_{1}=\{T\in \left[V\atop k\right]:{\rm dim}$$(T\cap B_{1})=0\}$, $\Gamma_{2}=\{T\in \left[V\atop k\right]:{\rm dim}(T\cap B_{2})=0, {\rm dim}(T\cap B_{1})>0\}$.  It is clear that $\Gamma_{1}\cup \Gamma_{2}\subseteq \Gamma(\{B_{1}, B_{2}\})$ and $\Gamma_{1}\cap \Gamma_{2}=\emptyset.$ Then $|\Gamma(\mathcal{B})|\geq |\Gamma_{1}|+|\Gamma_{2}|$.  By Lemma\;\ref{lemc21}~${\rm(i)}$, we have
\begin{equation}\label{81}
|\Gamma_{1}|= q^{k(k+1)}\left[n-k-1\atop k\right].
 \end{equation}

Obviously, there exists $E\in \left[{B}_{1}\atop 1\right], E\nleq {B}_{2}$ such that ${B}_{2}\leq W\in \left[V\atop n-1\right]$, where $V=E\oplus W$ (namely $V=E+W$ and $E\cap W=\{0\}$). Define a subfamily of $\Gamma_{2}$ by $\Gamma_{3}= \{T\in \left[V\atop k\right]:{\rm dim}(T\cap B_{2})=0, E\leq T\}$, which has a one-to-one correspondence to $\Gamma_{4}= \{R\in \left[W\atop k-1\right]:{\rm dim}(R\cap B_{2})=0\}$. There is a natural injective map $\varphi:\Gamma_{3}\rightarrow\Gamma_{4}$~by~$T\mapsto T\cap W$. For any $R\in \Gamma_{4}$, since $R+{B}_{2}\leq W, E\nleq W$, we have $E+R\in\Gamma_{3}$. So $\varphi$ is also surjective. By Lemma\;\ref{lemc21}~{\rm(i)}, we have
\begin{equation*}
|\Gamma_{2}|\geq |\Gamma_{3}|=|\Gamma_{4}|=q^{(k-1)(k+1)}\left[n-k-2\atop k-1\right].
 \end{equation*}
Together with \eqref{81}, we complete the proof of \eqref{eq31}. Combining \eqref{eq10},~\eqref{eq11},~and~\eqref{eq31}, we have
\begin{equation}\label{eq61}
\begin{array}{rl}
|\mathcal{A}|+|\mathcal{B}|\leq \left[n\atop k\right]-q^{k(k+1)}\left[n-k-1\atop k\right]-q^{(k-1)(k+1)}\left[n-k-2\atop k-1\right]+\left[n-2\atop k-1\right].
\end{array}
\end{equation}

Hence, we only need to show that
\begin{equation}\label{eq32}
\left[n-2\atop k-1\right]<1+q^{(k-1)(k+1)}\left[n-k-2\atop k-1\right]
\end{equation}
to prove the final conclusion.
If $k=2$, \eqref{eq32} clearly holds. If $k\geq 3$, by Lemma \ref{lemc21}~{\rm(ii)}, we have the inequality:
\begin{equation*}
q^{(k-1)(k+1)}\left[n-k-2\atop k-1\right]> \left[n-1\atop k-1\right]-\left[k+1\atop 1\right]\left[n-2\atop k-2\right].
 \end{equation*}
Since $q^{n-k}\leq \frac{q^{n}-1}{q^{k}-1}\leq q^{n-k+1}$ and $~n\geq 2k+2$, we have
\begin{equation*}
\begin{array}{rl}
\left[n-1\atop k-1\right]-\left[k+1\atop 1\right]\left[n-2\atop k-2\right]+1-\left[n-2\atop k-1\right]&> \left[n-1\atop k-1\right]-\left[k+1\atop 1\right]\left[n-2\atop k-2\right]-\left[n-2\atop k-1\right]\\[.3cm]
&= \frac{q^{n-1}-1}{q^{k-1}-1}\left[n-2\atop k-2\right]-\frac{q^{k+1}-1}{q-1}\left[n-2\atop k-2\right]-\frac{q^{n-k}-1}{q^{k-1}-1}\left[n-2\atop k-2\right]\\[.3cm]
&\geq (q^{n-k}-q^{k+1}-q^{n-2k+2})\left[n-2\atop k-2\right]\\[.3cm]
&\geq (2q^{n-k-1}-q^{k+1}-q^{n-2k+2})\left[n-2\atop k-2\right]\\[.3cm]
&\geq 0,\\[.3cm]
\end{array}
\end{equation*}
meaning that \eqref{eq32} holds as well. So we obtain that $|\mathcal{A}|+|\mathcal{B}|< \left[n\atop k\right]-q^{k(k+1)}\left[n-k-1\atop k\right]+1$ by \eqref{eq61}~and~\eqref{eq32}.

Reviewing the whole proof, we get that the families $\mathcal{A}$ and $\mathcal{B}$ attaining the equality in \eqref{eq6} are just those stated in the lemma.
\end{proof}

\section{ Proof of Theorem\;\ref{space two} }

In this section, we will prove Theorem\;\ref{space two}. Let $\mathcal{F}\subseteq \mathcal{L}(V)$ be an $s$-union family of maximum size with $\mathcal{F}\nsubseteq \mathcal{K}[n,s]$. We will calculate the cardinality
\begin{equation}\label{eq41}
|\mathcal{F}|=\sum _{0\leq i\leq s}|\mathcal{F}_i|, {\rm where} ~\mathcal{F}_i=\mathcal{F}\cap \left[V\atop i\right].
\end{equation}

We consider the singularity of $s$.

$(1)$ Let$~s=2d+1$. It is easily seen that $\mathcal{F}_{i}=\emptyset$ for $i\geq 2d+2$. Adding up \eqref{eq4} for $0\leq i\leq d$, we have
\begin{equation}\label{eq14}
\sum_{i=0}^d(|\mathcal{F}_i|+|\mathcal{F}_{2d+2-i}|)\leq \sum_{i=0}^d\left[n\atop i\right].
\end{equation}
Since $\mathcal{F}_{d+1}\subseteq \mathcal{F}$ is $(2d+1)$-union, we have that $\mathcal{F}_{d+1}$ is an intersecting family. Since for $G\leq F\in \mathcal{F}$, the family $\mathcal{F}\cup \{G\}$ is also $s$-union. So we can assume that $G\leq F$ and $F\in \mathcal{F}$ implies $G\in \mathcal{F}$. Then we distinguish two cases.

Case\ $a$$:$ Suppose there exists $G\in \mathcal{F}$ with ${\rm dim}(G)\geq d+2$. Then $\left[G\atop d+1\right]\subseteq \mathcal{F}_{d+1}$, which implies that ${\rm dim}(\bigcap_{F\in \mathcal{F}_{d+1}})=0$. By Theorem \ref{HM}, we have
\begin{equation}\label{eq7}
\begin{array}{rl}
|\mathcal{F}_{d+1}|\leq \left[n-1\atop d\right]-q^{d(d+1)}\left[n-d-2\atop d\right]+q^{d+1}.
\end{array}
\end{equation}
Moreover, there exists $\mathcal{F}_{2d+2-i}\neq \emptyset$ for $0\leq i\leq d$ in this case. Then $\eqref{eq4}$ is a strict inequality for some $i$, which implies the inequality in \eqref{eq14} is strict as well.
Hence, by \eqref{eq41}-\eqref{eq7}, we have
\begin{equation*}
|\mathcal{F}|=\sum_{i=0}^d(|\mathcal{F}_i|+|\mathcal{F}_{2d+2-i}|)+|\mathcal{F}_{d+1}|< \sum_{i=0}^d\left[n\atop i\right]+\left[n-1\atop d\right]-q^{d(d+1)}\left[n-d-2\atop d\right]+q^{d+1}.
\end{equation*}

Case\ $b$$:$ Suppose that $\mathcal{F}_i=\emptyset$ for all $i\geq d+2$. Now \eqref{eq4} holds trivially, that is \eqref{eq14} holds trivially as well. We have ${\rm dim}(\bigcap_{F\in \mathcal{F}_{d+1}})=0$, because otherwise $\mathcal{F}\subseteq \mathcal{K}[n,2d+1]$.
We use Theorem \ref{HM} again. The upper bound in \eqref{eq9} can be obtained by \eqref{eq14} and \eqref{eq7}. Moreover, if the equality in \eqref{eq9} holds, then $\mathcal{F}_i=\left[V\atop i\right]$ for all $0\leq i \leq d$ and the  equality in \eqref{eq7} holds as well. So applying Theorem \ref{HM}, we have
\begin{equation*}
\mathcal{F}_{d+1}=\{F\in \left[V\atop d+1\right]:E\leq F, {\rm dim}(F\cap U)\geq 1\}\bigcup \left[E+U\atop d+1\right],
\end{equation*}
where $ E, U$ are fixed $1$-dimensional, $(d+1)$-dimensional subspaces of $V$ with $E\nleq U$. If $d+1=3$, the equality is also attained by taking $\mathcal{F}_3=\{F\in \left[V\atop 3\right]:{\rm dim}(F\cap D)\geq 2\}$, where $D$ is a fixed $3$-dimensional subspace of $V$.

$(2)$ Let $s=2d$. For $0\leq i< d$, we use \eqref{eq4} as well. Then
\begin{equation}\label{eq16}
\sum _{0\leq i<d}(|\mathcal{F}_i|+|\mathcal{F}_{2d+1-i}|)\leq \sum _{0\leq i<d}\left[n\atop i\right].
\end{equation}

When $i=d$, we have a stronger result. For convenience, we set $\mathcal{A}=\mathcal{F}_d,~\mathcal{B}=\mathcal{F}_{d+1}$ in the following. Since $\mathcal{F}\nsubseteq \mathcal{K}[n,2d]$, there exists $G\in \mathcal{F}$ with ${\rm dim}(G)\geq d+1$. Hence $\left[G\atop d+1\right]\subseteq \mathcal{B}\neq \emptyset$.

Since $\mathcal{B}\subseteq \mathcal{F}$ is $2d$-union, then for all $B, B'\in \mathcal{B}$, we have \begin{equation*}
{\rm dim}(B\cap B')={\rm dim}B+{\rm dim}B'-{\rm dim}(B+B')\geq 2d+2-2d=2,
\end{equation*}
i.e., $\mathcal{B}$ is a 2-intersecting family. Similarly, we have that $\mathcal{A}$~and~$\mathcal{B}$ are cross-intersecting. Then we apply Lemma \ref{cross-} for $k=d$ to obtain

\begin{equation}\label{eq17}
|\mathcal{A}|+|\mathcal{B}|\leq \left[n\atop d\right]-q^{d(d+1)}\left[n-d-1\atop d\right]+1.
\end{equation}

Since for $i\geq 2d+1$, $\mathcal{F}_i=\emptyset$, we can add up \eqref{eq16} along with \eqref{eq17} to obtain \eqref{eq8} by \eqref{eq41}. Moreover, if the equality in \eqref{eq8} holds, then the  equalities in \eqref{eq16} and \eqref{eq17} hold as well. Then by Lemmas \ref{s,i} and \ref{cross-}, we have
\begin{equation*}
\mathcal{F}=\{F\leq V:{\rm dim}(F)\leq d\}\setminus\{F\subseteq \left[V\atop d\right]:{\rm dim}(F\cap U)=0\}\bigcup\{U\}.
\end{equation*}
where $U$ is a fixed $(d+1)$-dimensional subspace of $V$.
\qed

\section{ Proofs of Theorems\;\ref{2d+1} and\;\ref{antichain2} }
 In this section, we will prove Theorems\;\ref{2d+1} and\;\ref{antichain2}. The main approach adopts a series of replacement in an $s$-union antichain by shadows or shades. First, we will define the concept of shade and disclose a new relationship between a family of $k$-dimensional subspaces and its shadows or shades.

For a family $\mathcal{H}\subseteq \left[V\atop k\right]$, we define its {\it shade} by
\begin{equation*}
\bigtriangledown(\mathcal{H})=\{G\in \left[V\atop k+1\right]:H\leq G \text{~for~some}~H\in \mathcal{H}\}.
\end{equation*}

\begin{lem}\label{shade}
Suppose that $\mathcal{H}\subseteq \left[V\atop k\right]$, $n\geq 3$. Then the following hold.
 \begin{itemize}
\item[\rm(i)]~If $k\geq \lceil \frac{n}{2}\rceil+1$, then $|\bigtriangleup(\mathcal{H})|-|\mathcal{H}|\geq q\left[k-1\atop 1\right]$.  Moreover, equality holds if and only if $\mathcal{H}=\{U\}$, where $U$ is a fixed $k$-dimensional subspace of $V$.
 \item[\rm(ii)]~If $k\leq \lfloor \frac{n}{2}\rfloor-1$, then $|\bigtriangledown(\mathcal{H})|-|\mathcal{H}|\geq q\left[n-k-1\atop 1\right]$. Moreover, equality holds if and only if $\mathcal{H}=\{U\}$, where $U$ is a fixed $k$-dimensional subspace of $V$.
\end{itemize}

\end{lem}

\begin{proof}~
\rm(i) Let $|\mathcal{H}|=\left[x\atop k\right]$, where $k\leq x\leq n\leq 2k-2$. Then by Theorem \ref{k-1-shadow}, $|\bigtriangleup(\mathcal{H})|\geq \left[x\atop k-1\right]$. So we have
\begin{equation*}
\begin{array}{rl}
|\bigtriangleup(\mathcal{H})|-|\mathcal{H}|&\geq \left[x\atop k-1\right]-\left[x\atop k\right]
=\frac{q^{k}-q^{x-k+1}}{q^{k}-1}\left[x\atop k-1\right]
=\frac{q^{k}-q^{x-k+1}}{q^{k}-1}\prod_{i=0}^{k-2}\frac{q^{x-i}-1}{q^{k-1-i}-1}.
\end{array}
\end{equation*}
Let $f(x)$ be the $\text{RHS}$ of the above inequality. By setting $y=q^{x}$, we can rewrite $f(x)$ as a polynomial $g(y)$ of degree $k$, namely
\begin{equation*}
\begin{array}{rl}
g(y)= \frac{q^{k}-yq^{-k+1}}{q^{k}-1}\prod_{i=0}^{k-2}\frac{yq^{-i}-1}{q^{k-1-i}-1}.
\end{array}
\end{equation*}
Because the polynomial $g(y)$ has $k$ simple roots $1,~q,~\ldots,~q^{k-2},~q^{2k-1}$, $f(x)$ has $k$ simple roots $0,~1,~\ldots,~k-2,~2k-1.$ It is clear that $f'(x)$ has a simple root in each interval between these roots. Since $f(x)<0$~if~$x> 2k-1$, then in $[k-2, 2k-1]$,~$f(x)$ is  increasing up to some value and then decreasing. Hence for $k\leq x\leq n\leq 2k-2$,

\begin{equation*}
f(x)\geq \min\{f(k),~f(2k-2)\}.
\end{equation*}
Clearly, we have
\begin{equation*}
\begin{array}{rl}
f(2k-2)-f(k)&=\left[2k-2\atop k-1\right]-\left[2k-2\atop k\right]-\left[k\atop k-1\right]+1\\[.3cm]
&=\frac{q^{k}-q^{k-1}}{q^{k-1}-1}\left[2k-2\atop k-2\right]-q\left[k-1\atop 1\right]\\[.3cm]
&>\left[2k-2\atop k-2\right]-q^{k}\\[.3cm]
&>0.\\[.3cm]
\end{array}
\end{equation*}
Therefore,
\begin{equation*}
f(x)\geq f(k)=q\left[k-1\atop 1\right].
\end{equation*}
The equality holds if and only if $x=k$, that is $|\mathcal{H}|=1$. Equivalently,~$\mathcal{H}=\{U\}$, where $U$ is a fixed $k$-dimensional subspace of $V$.

\rm(ii) We claim that $|\bigtriangledown(\mathcal{H})|=|\bigtriangleup(\mathcal{H}^{\bot})|$. For any $G\in \bigtriangledown(\mathcal{H})$, there exists $H\in \mathcal{H}$ such that $H\leq G$. Then $G^{\bot}\leq H^{\bot}$, where ${\rm dim}(G^{\bot})=n-k-1,~{\rm dim}(H^{\bot})=n-k$, and$~H^{\bot}\in \mathcal{H}^{\bot}$. Thus $G^{\bot}\in \bigtriangleup(\mathcal{H}^{\bot})$.
This gives an injective map $\varphi:\bigtriangledown(\mathcal{H})\rightarrow \bigtriangleup(\mathcal{H}^{\bot})$ by $G\mapsto G^{\bot}$, which is obviously surjective.

Let $k\leq \lfloor \frac{n}{2}\rfloor-1$. Then $n-k\geq \lceil\frac{n}{2}\rceil+1$. We can obtain the following inequality by the above claim and the result of ${\rm(i)}$:
\begin{equation*}
|\bigtriangledown(\mathcal{H})|-|\mathcal{H}|=|\bigtriangleup(\mathcal{H}^{\bot})|-|\mathcal{H}^{\bot}|\geq q\left[n-k-1\atop 1\right].
\end{equation*}
Moreover, the equality holds if and only if $\mathcal{H^{\bot}}=\{U^{\bot}\}$, that is $\mathcal{H}=\{U\}$, where $U$ is a fixed $k$-dimensional subspace of $V$.
\end{proof}

Throughout the remainder of the section, let $\mathcal{F}\subseteq \mathcal{L}(V)$ be a given $s$-union antichain and let $d=\lfloor\frac{s}{2}\rfloor$. For any family $\mathcal{G}\subseteq \mathcal{L}(V)$, define $$l(\mathcal{G})=\max\{{\rm dim}(G):G\in \mathcal{G}\},$$ $$m(\mathcal{G})=\min\{{\rm dim}(G):G\in \mathcal{G}\}.$$ For short $l(\mathcal{F})$, $m(\mathcal{F})$ are briefly denoted by $l$ and $m$ respectively. Denote $~\mathcal{F}_i=\mathcal{F}\cap \left[V\atop i\right]$ for $m\leq i\leq l$.

\begin{lem}\label{4.1}
If $l\leq d $, then the following hold.
\begin{itemize}
\item[\rm(i)]~$|\mathcal{F}|\leq \left[n\atop d\right]$. Moreover, equality holds if and only if $\mathcal{F}=\left[V\atop d\right]$.
\item[\rm(ii)]~If $\mathcal{F}\nsubseteq \left[V\atop d\right]$, then
$|\mathcal{F}|\leq \left[n\atop d\right]-q\left[n-d\atop 1\right]$.
 Moreover, equality holds if and only if $\mathcal{F}=(\left[V\atop d\right]
 \setminus\{F\in \left[V\atop d\right]:U\leq F\})\bigcup \{U\}$, where $U$ is a fixed $(d-1)$-dimensional subspace of $V$.
\end{itemize}
\end{lem}

\begin{proof}~
 $(1)$~Suppose $m=l<d.$ Then $|\mathcal{F}|\leq \left[n\atop l\right]\leq \left[n\atop d-1\right]$. If $d=1$, $|\mathcal{F}|=1=\left[n\atop 1\right]-q\left[n-1\atop 1\right]$. If $d\geq2$, we have
 \begin{equation*}
 \begin{array}{rl}
 \left[n\atop d\right]-q\left[n-d\atop 1\right]-\left[n\atop d-1\right]&=\frac{q^{n-d+1}-q^{d}}{q^{d}-1}\left[n\atop d-1\right]-q\left[n-d\atop 1\right]\\[.3cm]
 &>\left[n\atop d-1\right]-q\left[n-d\atop 1\right]\\[.3cm]
 &>\frac{q^{n}-1}{q^{d-1}-1}-q^{n-d+1}\\[.3cm]
 &>0.\\[.3cm]
 \end{array}
\end{equation*}
Hence,

\begin{equation*}
|\mathcal{F}|\leq \left[n\atop d-1\right]<\left[n\atop d\right]-q\left[n-d\atop 1\right].
\end{equation*}

$(2)$~Suppose $m=l=d.$~Then $|\mathcal{F}|\leq \left[n\atop d\right]$, and equality holds if and only if $\mathcal{F}=\left[V\atop d\right]$.

$(3)$~Suppose $m< l\leq d.$ Let $\mathcal{H}=\mathcal{F}_{m}$,~$\mathcal{G}_{1}=(\mathcal{F}\setminus \mathcal{H})\cup \bigtriangledown(\mathcal{H})$. It is clear that $\mathcal{G}_{1}$ is also an $s$-union antichain in this case and $(\mathcal{F}\setminus \mathcal{H})\cap \bigtriangledown(\mathcal{H})=\emptyset$. Then by Lemma \ref{shade}, we have
\begin{equation}\label{eq15}
|\mathcal{G}_{1}|=|\mathcal{F}|-|\mathcal{H}|+|\bigtriangledown(\mathcal{H})|\geq |\mathcal{F}|+q\left[n-m-1\atop 1\right].
\end{equation}

If $m(\mathcal{G}_{1})=m+1< l$, then let $\mathcal{H}_{1}=\mathcal{G}_{1}\cap \left[V\atop m(\mathcal{G}_{1})\right]$ and $\mathcal{G}_{2}=(\mathcal{G}_{1}\setminus \mathcal{H}_{1})\cup \bigtriangledown(\mathcal{H}_{1})$. After this we obtain an $s$-union antichain $\mathcal{G}_{2}$ for which by Lemma \ref{shade}

\begin{equation*}
\begin{array}{rl}
|\mathcal{G}_{2}|&=|\mathcal{G}_{1}|-|\mathcal{H}_{1}|+|\bigtriangledown(\mathcal{H}_{1})|\\[.3cm]
&\geq |\mathcal{G}_{1}|+q\left[n-m(\mathcal{G}_{1})-1\atop 1\right]\\[.3cm]
&\geq |\mathcal{F}|+q\left[n-m-1\atop 1\right]+q\left[n-m-2\atop 1\right].\\[.3cm]
\end{array}
\end{equation*}

Repeat doing like this until we raise the minimum dimension of the spaces in $\mathcal{F}$ to $l$, and we obtain an $s$-union antichain $\mathcal{G}_{l-m}$ satisfying

\begin{equation*}
|\mathcal{G}_{l-m}|\geq|\mathcal{F}|+q\left[n-m-1\atop 1\right]+q\left[n-m-2\atop 1\right]+\cdots+q\left[n-l\atop 1\right].
\end{equation*}
Since $\mathcal{G}_{l-m}\subseteq \left[V\atop l\right]$,~$l\leq d$, then $|\mathcal{G}_{l-m}|\leq \left[n\atop d\right]$. Since $q\left[n-l\atop 1\right]\leq q\left[n-i\atop 1\right]$ for $m+1\leq i\leq l$ and $m<l\leq d$, we have
\begin{equation}\label{eq18}
\begin{array}{rl}
|\mathcal{F}|&\leq\left[n\atop d\right]-q\left[n-m-1\atop 1\right]-q\left[n-m-2\atop 1\right]-\cdots-q\left[n-l\atop 1\right]\\[.3cm]
&\leq\left[n\atop d\right]-q\left[n-m-1\atop 1\right]\\[.3cm]
&\leq\left[n\atop d\right]-q\left[n-l\atop 1\right]\\[.3cm]
&\leq\left[n\atop d\right]-q\left[n-d\atop 1\right].
\end{array}
\end{equation}
Moreover, equality holds if and only if $m+1= l=d$,~$\mathcal{G}_{1}=\left[V\atop d\right]$~and the equality in \eqref{eq15} holds. That is $\mathcal{H}=\{U\}$, where $U$ is a fixed $(d-1)$-dimensional subspace of $V$  by Lemma \ref{shade}. Thus

\begin{equation*}
\mathcal{F}=(\mathcal{G}_{1}\setminus \bigtriangledown(\mathcal{H}))\bigcup \mathcal{H}=(\left[V\atop d\right]\setminus\{F\in \left[V\atop d\right]:U\leq  F\})\bigcup \{U\}.
\end{equation*}

Combining (1)-(3) yields part (i) of the lemma. Note that if $\mathcal{F}\nsubseteq \left[V\atop d\right]$ then $m=l<d$ or $m< l\le d$ and the arguments in (1) and (3) apply. This proves part (ii).
\end{proof}

\begin{lem}\label{4.2}
If~$l\geq d+1$, then the following hold.
\begin{itemize}
\item[\rm(i)]~If $m<d<l$, then $|\mathcal{F}|\leq\left[n\atop d\right]-q\left[n-d\atop 1\right]$ and equality holds only if $s$ is odd.

\item[\rm(ii)]~If $m=d<l$, then
\begin{equation*}
|\mathcal{F}|\!\leq\!
\begin{cases}
\left[n\atop d\right]-q^{d(d+1)}\left[n-d-1\atop d\right]+1, &\text{if~$s=2d<n$},\\[.3cm]
\left[n\atop d\right], &\text{if~$s=2d+1\leq n$}.
\end{cases}\
\end{equation*}
\item[\rm(iii)]~If $d<m\leq l$, then
\begin{equation*}
|\mathcal{F}|\!\leq\!
\begin{cases}
\left[s\atop d+1\right], &\text{if~$n\leq 2m$},\\[.3cm]
\left[n+s-2d-2\atop s-d-1\right], &\text{if~$n> 2m$}.
\end{cases}\
\end{equation*}
\end{itemize}

\end{lem}

\begin{proof}~
Let $\mathcal{D}=\mathcal{F}_{l}$, $\mathcal{F}^{1}=(\mathcal{F}\setminus \mathcal{D})\cup \bigtriangleup(\mathcal{D})$. It is obvious that $\mathcal{F}^{1}$ is also an $s$-union antichain and $(\mathcal{F}\setminus \mathcal{D})\cap \bigtriangleup(\mathcal{D})=\emptyset$. Since $\mathcal{F}$ is an $s$-union antichain and $l\geq d+1$, then for any $D, D'\in \mathcal{D}$, we  have
\begin{equation*}
{\rm dim}(D\cap D')=2l-{\rm dim}(D+D')\geq 2l-s\geq1.
\end{equation*}
By Theorem~\ref{shadow intersecting}, we have $|\bigtriangleup(\mathcal{D})|\geq|\mathcal{D}|$. Then

\begin{equation*}
|\mathcal{F}^{1}|=|\mathcal{F}|-|\mathcal{D}|+|\bigtriangleup(\mathcal{D})|\geq|\mathcal{F}|.
\end{equation*}
Note that the inequality is strict if either $s=2d$ or $s=2d+1$ and $l\geq d+2$.
If $l(\mathcal{F}^{1})=l-1\geq\max\{d+1,~m+1\}$, then let $\mathcal{D}_{1}=\mathcal{F}^{1}\cap \left[V\atop l(\mathcal{F}^{1})\right]$, $\mathcal{F}^{2}=(\mathcal{F}^{1}\setminus \mathcal{D}_{1})\cup \bigtriangleup(\mathcal{D}_{1})$. After this we obtain an $s$-union antichain $\mathcal{F}^{2}$ for which by Theorem \ref{shadow intersecting}

\begin{equation*}
|\mathcal{F}^{2}|=|\mathcal{F}^{1}|-|\mathcal{D}_{1}|+|\bigtriangleup(\mathcal{D}_{1})|\geq |\mathcal{F}^{1}|\geq|\mathcal{F}|.
\end{equation*}


${\rm(i)}$~Suppose $m<d<l$. Repeat the above process until we decrease the maximum dimension of the spaces in $\mathcal{F}$ to $d$.

If $s=2d$, we obtain a $2d$-union antichain $\mathcal{F}^{l-d}$ satisfying
\begin{equation*}
|\mathcal{F}^{l-d}|> |\mathcal{F}^{l-d-1}|> \cdots >|\mathcal{F}^{1}|>|\mathcal{F}|.
\end{equation*}
Since $m=m(\mathcal{F}^{l-d})<l(\mathcal{F}^{l-d})=d<l$, then replacing $\mathcal{F}$ with $\mathcal{F}^{l-d}$ in (3) of the proof of Lemma \ref{4.1} and applying  \eqref{eq18}, we have
\begin{equation*}
|\mathcal{F}|<|\mathcal{F}^{l-d}|\leq \left[n\atop d\right]-q\left[n-d\atop 1\right].
\end{equation*}

If $s=2d+1$, we obtain a $(2d+1)$-union antichain $\mathcal{F}^{l-d}$ satisfying

\begin{equation*}
|\mathcal{F}^{l-d}|\geq|\mathcal{F}^{l-d-1}|> \cdots >|\mathcal{F}^{1}|>|\mathcal{F}|.
\end{equation*}
Here note that if $l=d+1$, then we have $|\mathcal{F}^{l-d}|=|\mathcal{F}^{1}|\geq|\mathcal{F}|$.
Similarly by \eqref{eq18}, we have
\begin{equation*}
|\mathcal{F}|\leq|\mathcal{F}^{l-d}|\leq \left[n\atop d\right]-q\left[n-d\atop 1\right].
\end{equation*}

${\rm(ii)}$~Suppose $m=d<l$. Similarly as above but we decrease the maximum dimension of the spaces in $\mathcal{F}$ to $d+1$ if $l>d+1$. Then we obtain an $s$-union antichain $\mathcal{F}^{l-d-1}$ satisfying

\begin{equation*}
|\mathcal{F}^{l-d-1}|> |\mathcal{F}^{l-d-2}|> \cdots >|\mathcal{F}^{1}|>|\mathcal{F}|.
\end{equation*}
If $l=d+1$, then let $\mathcal{F}^{l-d-1}=\mathcal{F}$.
It is clear that $m(\mathcal{F}^{l-d-1})=d,~l(\mathcal{F}^{l-d-1})=d+1$.~Let $\mathcal{H}=\mathcal{F}^{l-d-1}\cap \left[V\atop d\right]$,~$\mathcal{G}=\mathcal{F}^{l-d-1}\cap \left[V\atop d+1\right]$.

If $s=2d$, then by the $2d$-union property of $\mathcal{F}^{l-d-1}$,
we have that $\mathcal{H}$ and $\mathcal{G}$ are cross-intersecting families and $\mathcal{G}$ is a $2$-intersecting family. Whenever $n\geq 2d+1$, by Lemma~\ref{cross-}, we have
\begin{equation*}
|\mathcal{F}|\leq|\mathcal{F}^{l-d-1}|=|\mathcal{H}|+|\mathcal{G}|\leq \left[n\atop d\right]-q^{d(d+1)}\left[n-d-1\atop d\right]+1.
\end{equation*}

If $s=2d+1$, then by the $(2d+1)$-union property of $\mathcal{F}^{l-d-1}$, we have that $\mathcal{G}$ is an intersecting family. By Theorem \ref{shadow intersecting}, we have
\begin{equation}\label{80}
|\mathcal{F}|\leq|\mathcal{F}^{l-d-1}|=|\mathcal{H}|+|\mathcal{G}|\leq\left[n\atop d\right]-|\bigtriangleup(\mathcal{G})|+|\mathcal{G}|\leq \left[n\atop d\right].
\end{equation}

${\rm(iii)}$~Suppose $m>d$. As before but we decrease the maximum dimension of the spaces in $\mathcal{F}$ to $m$ if $l>m$. Now we obtain an $s$-union antichain $\mathcal{F}^{l-m}\subseteq \left[V\atop m\right]$ satisfying

\begin{equation*}
|\mathcal{F}^{l-m}|> |\mathcal{F}^{l-m-1}|> \cdots >|\mathcal{F}^{1}|>|\mathcal{F}|.
\end{equation*}
If $l=m$, then let $\mathcal{F}^{l-m}=\mathcal{F}$.
By the $s$-union property of $\mathcal{F}^{l-m}$, for any $F, F'\in \mathcal{F}^{l-m}$, we  have ${\rm dim}(F\cap F')=2m-{\rm dim}(F+F')\geq 2m-s.$

If $n\leq 2m$, then by Theorem~\ref{EKR} and noting $m>d$,  we have

\begin{equation*}
|\mathcal{F}^{l-m}|\leq \left[s\atop m\right]\leq \left[s\atop d+1\right].
\end{equation*}

If $n>2m$, then by Theorem~\ref{EKR} and noting $s\geq l\geq m>d$,  we have

\begin{equation*}
|\mathcal{F}^{l-m}|\leq \left[n+s-2m\atop s-m\right]\leq \left[n+s-2d-2\atop s-d-1\right].
\end{equation*}
\end{proof}

\noindent\emph{{\textbf{Proof of Theorem\;{\rm\ref{2d+1}.}}}}~
Obviously, we have the two assertions in Lemma \ref{4.1} for the case $l\leq d$. Next, we give new upper bounds of $|\mathcal{F}|$ for the case $l>d$ by similar approach of Lemma \ref{4.2}. Now $d=\lfloor\frac{n}{2}\rfloor$.

{\rm(1)}~Suppose $m<d<l$. By Lemma \ref{4.2} ${\rm(i)}$, if $n=2d$, then $|\mathcal{F}|<\left[n\atop d\right]-q\left[d\atop 1\right]$; if $n=2d+1$, $|\mathcal{F}|\leq\left[n\atop d\right]-q\left[n-d\atop 1\right]<\left[n\atop d\right]-q\left[d\atop 1\right].$

{\rm(2)}~Suppose $m=d<l$. Similarly as the proof of Lemma~\ref{4.2}~{\rm(ii)},
we obtain an antichain $\mathcal{F}^{l-d-1}$ satisfying $|\mathcal{F}|\leq|\mathcal{F}^{l-d-1}|,$
$m(\mathcal{F}^{l-d-1})=d$ and $l(\mathcal{F}^{l-d-1})=d+1$.~Let ~$\mathcal{G}=\mathcal{F}^{l-d-1}\cap \left[V\atop d+1\right]$.

Case\ $a$$:$ If $n=2d$, let $\mathcal{M}=(\mathcal{F}^{l-d-1}\setminus \mathcal{G})\cup \bigtriangleup(\mathcal{G})$. Then we obtain an antichain $\mathcal{M}$ satisfying~

\begin{equation*}
|\mathcal{M}|\leq \left[n\atop d\right].
\end{equation*}
By Lemma~\ref{shade}~${\rm(i)}$, we have

\begin{equation*}
|\bigtriangleup(\mathcal{G})|-|\mathcal{G}|\geq q\left[d\atop 1\right].
\end{equation*}
Hence,
\begin{equation*}
|\mathcal{F}|\leq|\mathcal{F}^{l-d-1}|= |\mathcal{M}|+|\mathcal{G}|-|\bigtriangleup(\mathcal{G})|\leq\left[n\atop d\right]-q\left[d\atop 1\right].
\end{equation*}
Moreover, the equality holds if and only if~$l=d+1,~\mathcal{M}=\left[V\atop d\right]$,
~$\mathcal{G}=\{W\}$, where $W$ is a fixed $(d+1)$-dimensional subspace of $V$, that is
\begin{equation*}
\mathcal{F}=(\mathcal{M}\setminus \bigtriangleup(\{W\})\bigcup \{W\}=\mathcal{B}[n,n].
\end{equation*}

Case\ $b$$:$ If $n=2d+1$, let $\mathcal{H}=\mathcal{F}^{l-d-1}\cap \left[V\atop d\right]$. Since $\mathcal{F}^{l-d-1}$ is an antichain, we have that $\mathcal{H}, \mathcal{G}$ are cross-Sperner. Then by Theorem \ref{cross-sperner},
\begin{equation*}
|\mathcal{F}|\leq |\mathcal{F}^{l-d-1}|=|\mathcal{H}|+|\mathcal{G}|\leq\left[n\atop d\right]-q\left[d\atop 1\right].
\end{equation*}
Moreover, equality holds if and only if  $l=d+1$ and either $\mathcal{F}=\mathcal{A}[n,n]$ or $\mathcal{F}=\mathcal{B}[n,n]$.

{\rm(3)}~Suppose $d<m\leq l$.

Case\ $a$$:$ Let $n=2d$. Now we have $n<2m$. Then by Lemma~\ref{4.2}~{\rm{(iii)}}, $|\mathcal{F}|\leq \left[2d\atop d+1\right]<\left[2d\atop d\right]-q\left[d\atop 1\right].$

Case\ $b$$:$ Let $n=2d+1$.

 If $d<m=l$, then $|\mathcal{F}|\leq \left[2d+1\atop m\right]\leq \left[2d+1\atop d+1\right].$ Moreover, equality holds if and only if $\mathcal{F}=\left[V\atop d+1\right]$. Further, if $\mathcal{F}\nsubseteq\left[V\atop d+1\right]$, then $|\mathcal{F}|\leq \left[2d+1\atop d+2\right]<\left[2d+1\atop d\right]-q\left[d\atop 1\right].$

If $d<m<l$, similarly as the proof of Lemma~\ref{4.2}~{\rm(iii)} but we decrease the maximum dimension of the spaces in $\mathcal{F}$ to $m+1$ if $l>m+1$, then we obtain an antichain $\mathcal{F}^{l-m-1}$ satisfying

\begin{equation*}
|\mathcal{F}^{l-m-1}|> \cdots >|\mathcal{F}^{1}|>|\mathcal{F}|.
\end{equation*}
If $l=m+1$, then let $\mathcal{F}^{l-m-1}=\mathcal{F}$. ~Let ~$\mathcal{G}=\mathcal{F}^{l-m-1}\cap \left[V\atop m+1\right]$ and $\mathcal{N}=(\mathcal{F}^{l-m-1}\setminus \mathcal{G})\cup \bigtriangleup(\mathcal{G})$. By Lemma~\ref{shade}~${\rm(i)}$, we have

\begin{equation*}
|\bigtriangleup(\mathcal{G})|-|\mathcal{G}|\geq q\left[m\atop 1\right].
\end{equation*}
Hence,
\begin{equation*}
|\mathcal{F}|\leq|\mathcal{F}^{l-m-1}|= |\mathcal{N}|+|\mathcal{G}|-|\bigtriangleup(\mathcal{G})|\leq\left[2d+1\atop m\right]-q\left[m\atop 1\right]<\left[2d+1\atop d\right]-q\left[d\atop 1\right].
\end{equation*}

To sum up, an optimal antichain satisfies $|\mathcal{F}|\leq \left[n\atop \lfloor\frac{n}{2}\rfloor\right]$ and equality occurs in  {\rm(3)} Case\ $b$ or Lemma~\ref{4.1}~${\rm(i)}$; a suboptimal antichain has $|\mathcal{F}|\leq \left[n\atop \lfloor\frac{n}{2}\rfloor\right]-q\left[\lfloor\frac{n}{2}\rfloor\atop 1\right]$ and equality occurs in  {\rm(2)} or Lemma~\ref{4.1}~${\rm(ii)}$ if $n=2d$. This completes the proof.
\qed

\noindent\emph{{\textbf{Proof of Theorem\;{\rm\ref{antichain2}}.}}}~
We divide the proof into two parts, according to the singularity of $s$.

$(1)$~Suppose $s=2d$. The case of $d=1$ is trivial and has been explained in Section ${\rm 1}$. Hence, we only need to consider $d\geq 2$ in the following.
By Lemmas \ref{4.1} and  \ref{4.2}, it suffices to show that the upper bounds provided in ${\rm(ii)}$ and ${\rm(iii)}$ of Lemma \ref{4.2} are strictly smaller than $\left[n\atop d\right]-q\left[n-d\atop 1\right]$.

Case\ $a$$:$~It is readily checked that
\begin{equation*}
\begin{array}{rl}
q^{d(d+1)}\left[n-d-1\atop d\right]-1-q\left[n-d\atop 1\right]
&=q^{d(d+1)}\left[n-d-1\atop d\right]-\left[n-d+1\atop 1\right]\\[.3cm]
&\geq q^{d(d+1)}\left[n-d-1\atop d\right]-q^{n-d+1}\\[.3cm]
&\geq\frac {q^{d(d+1)}(q^{n-d-1}-1)}{q^{d}-1}
-q^{n-d+1}\\[.3cm]
&\geq\frac {q^{2d+2}(q^{n-d-1}-1)}{q^{d}-1}
-q^{n-d+1}\\[.3cm]
&> q^{d+2}(q^{n-d-1}-1)-q^{n-d+1}\\[.3cm]
&>0.\\[.3cm]
\end{array}
\end{equation*}
Hence,

\begin{equation*}
\left[n\atop d\right]-q^{d(d+1)}\left[n-d-1\atop d\right]+1<\left[n\atop d\right]-q\left[n-d\atop 1\right].
\end{equation*}

Case\ $b$$:$~Noting $n>s=2d, \left[n\atop d-1\right]>q^{n-d+1}>q\left[n-d\atop 1\right]$, we have
\begin{equation*}
\begin{array}{rl}
\left[n\atop d\right]-q\left[n-d\atop 1\right]-\left[2d\atop d+1\right]
&=\frac{q^{n-d+1}-1}{q^{d}-1}\left[n\atop d-1\right]-\left[2d\atop d-1\right]-q\left[n-d\atop 1\right]\\[.3cm]
&>\frac{q^{n-d+1}-q^{d}}{q^{d}-1}\left[n\atop d-1\right]-q\left[n-d\atop 1\right]\\[.3cm]
&>0.\\[.3cm]
\end{array}
\end{equation*}
Hence,
\begin{equation*}
\left[2d\atop d+1\right] <\left[n\atop d\right]-q\left[n-d\atop 1\right].
\end{equation*}

Case\ $c$$:$~Since $\left[a\atop k\right]=q^{a-k}\left[a-1\atop k-1\right]+\left[a-1\atop k\right]$ for $a\geq k+1$ and note $d\geq 2$, we have

\begin{equation*}
\begin{array}{rl}
\left[n\atop d\right]-q\left[n-d\atop 1\right]-\left[n-2\atop d-1\right]
&=q^{n-d}\left[n-1\atop d-1\right]+\left[n-1\atop d\right]-q\left[n-d\atop 1\right]-\left[n-2\atop d-1\right]\\[.3cm]
&=q^{n-d}\left[n-1\atop d-1\right]+q^{n-d-1}\left[n-2\atop d-1\right]+\left[n-2\atop d\right]-q\left[n-d\atop 1\right]-\left[n-2\atop d-1\right]\\[.3cm]
&>q^{n-d}\left[n-1\atop d-1\right]+\left[n-2\atop d\right]-q\left[n-d\atop 1\right]\\[.3cm]
&>0.\\[.3cm]
\end{array}
\end{equation*}
Hence,
\begin{equation*}
\left[n-2\atop d-1\right] <\left[n\atop d\right]-q\left[n-d\atop 1\right].
\end{equation*}

$(2)$~Suppose $s=2d+1$. It is clear that the upper bounds provided in Lemma \ref{4.2}~{\rm{(i)}} and~{\rm{(iii)}} are strictly smaller than $\left[n\atop d\right]$; and the equality in Lemma \ref{4.2}~{\rm{(ii)}} holds if and only if \eqref{80} holds, that is $\mathcal{F}=\mathcal{H}\bigcup \mathcal{G},$ where $\mathcal{G}\subseteq\left[V\atop d+1\right], \mathcal{H}=\left[V\atop d\right]\setminus \bigtriangleup(\mathcal{G})$ and $ |\bigtriangleup(\mathcal{G})|=|\mathcal{G}|$. Then by Lemma \ref{4.1}, we complete the proof.
\qed
\section{ Concluding remarks }

In the present paper, we determine all suboptimal $s$-union families for vector spaces.
For $s=n$ or $s=2d<n$, we determine all optimal and suboptimal $s$-union antichains completely. For $s=2d+1<n$, we prove that an optimal $s$-union antichain is either $\left[V\atop d\right]$ or
$\mathcal{F}=\mathcal{F}_{d}\bigcup \mathcal{F}_{d+1}$, where $\mathcal{F}_{d+1}\subseteq \left[V\atop d+1\right], \mathcal{F}_{d}=\left[V\atop d\right]\setminus \bigtriangleup(\mathcal{F}_{d+1})$ and
$|\bigtriangleup(\mathcal{F}_{d+1})|=|\mathcal{F}_{d+1}|$. It is very interesting to display all optimal $(2d+1)$-union antichains of the latter type. Obviously,
${\cal F}=(\left[V\atop d\right]\setminus\left[S\atop d\right])\bigcup\left[S\atop d+1\right]$ satisfies the above condition, where $S$ is a fixed $(2d+1)$-dimensional subspace of $V$.
Consulting the situation of $s$-union antichains in an $n$-element set, we conjecture that this is the unique desired structure (of the latter type).

Let $\mathcal{F}\subseteq \mathcal{L}(V)$ be a $(2d+1)$-union antichain with $\mathcal{F}$ not contained in any optimal $(2d+1)$-union antichain. By Lemma \ref{4.1}, if $l\leq d$, we have
\begin{equation*}
|\mathcal{F}|\leq \left[n\atop d\right]-q\left[n-d\atop 1\right]<\left[n\atop d\right]-q\left[d\atop 1\right].
\end{equation*}
It is obvious that the upper bounds provided in Lemma \ref{4.2}~{\rm{(i)}} and~{\rm{(iii)}} are strictly smaller than $\left[n\atop d\right]-q\left[d\atop 1\right]$. From the proof of  Lemma \ref{4.2}~{\rm{(ii)}},
we know that a suboptimal $(2d+1)$-union antichain has the form $\mathcal{F}=\mathcal{F}_{d}\bigcup \mathcal{F}_{d+1}$, where $\mathcal{F}_{d+1}\subseteq \left[V\atop d+1\right],
\mathcal{F}_{d}=\left[V\atop d\right]\setminus \bigtriangleup(\mathcal{F}_{d+1})$ and $|\bigtriangleup(\mathcal{F}_{d+1})|>|\mathcal{F}_{d+1}|$. In view of Theorem \ref{2d+1}, we make the following conjecture.






\begin{conj}\label{con}
Let $\mathcal{F}\subseteq \mathcal{L}(V)$ be a $(2d+1)$-union antichain with $\mathcal{F}$ not contained in any optimal $(2d+1)$-union antichain and $2d+1< n$. Then
\begin{equation*}
|\mathcal{F}|\leq \left[n\atop d\right]-q\left[d\atop 1\right].
\end{equation*}
 Moreover, equality holds if and only if $\mathcal{F}=\mathcal{B}[n,2d+1]$.
 \end{conj}

\end{document}